\documentclass[11pt]{amsart}
\usepackage{latexsym}
\usepackage{amsmath,amssymb,amsfonts}
\usepackage{graphics}
\usepackage[all]{xy}

\def\Bbb{\mathbb}

\def\eea{\end{eqnarray*}}

\newtheorem{defn}{Definition}
\newtheorem{thm}{Theorem}[section]

\newtheorem{lem}[thm]{Lemma}

\newenvironment{rmk}{\mbox{ }\\{\bf  Remark}\mbox{ }}{
\hfill $\Box$\mbox{}\bigskip}

\begin{document}
\renewcommand{\theequation}{\thesection.\arabic{equation}}

\Large
\begin{center}
\textbf{Some refined higher type adjunction inequalities on 4-manifolds}
\end{center}
\vspace{4mm}
\normalsize


\begin{center}
\textbf{Chanyoung Sung}\\
\vspace{5mm} \small{\itshape Department of Mathematics Educatuon\\ Korea National University of Education, Cheongju, Republic of Korea}
\end{center}

\hrulefill \vspace{2mm}

 \small {\textbf{Abstract.}  

We further sharpen higher type adjunction inequalities of P. Ozsv\'ath and Z. Szab\'o on a 4-manifold $M$ with a nonzero Seiberg-Witten invariant for a Spin$^c$ structure $\frak{s}$, when an embedded surface $\Sigma\subset M$ satisfies $[\Sigma]\cdot [\Sigma]\geq 0$ and
$$|\langle [\Sigma],c_1(\frak{s})\rangle|+[\Sigma]\cdot [\Sigma]\geq 2b_1(M).$$

 \vspace{2mm}


 }

\hrulefill \normalsize \hrulefill \normalsize

\footnotetext[1]{Date : \today}

\footnotetext[2]{E-mail address : cysung@kias.re.kr}
\thispagestyle{empty}

\setcounter{section}{0}
\setcounter{equation}{0}

\section{Introduction}

Given a Spin$^c$ structure $\frak{s}$ on a smooth closed oriented Riemannian 4-manifold $M$, for a section $\Phi$ of the plus spinor bundle $W_+$ of $\frak{s}$ and a $u(1)$ connection $A$ on $\det W_+$, the Seiberg-Witten equations are given by
$$\left\{
\begin{array}{ll} D_A\Phi=0\\
  F_{A}^+ +i\eta= \Phi\otimes\Phi^*-\frac{|\Phi|^2}{2}\textrm{Id},
\end{array}\right.
$$
where $D_A$ and $F_A^+$ respectively denote the associated Dirac operator and the self-dual part of the curvature $dA$ of $A$, a self-dual $2$-form $\eta$ is a generic perturbation term, and lastly the identification of both sides in the second equation comes from the Clifford action.

Its moduli space $\frak{M}$, i.e. the space of solutions modulo bundle automorphisms known as the gauge group $\mathcal{G}:=Map(M,S^1)$ is a smooth orientable manifold of dimension $$d(\frak{s}):=\frac{c_1(\frak{s})^2-(2\chi(M)+3\tau(M))}{4},$$ where $\chi$ and $\tau$ denote  Euler characteristic and signature respectively.

The intersection theory on $\frak{M}$ produces Seiberg-Witten invariants in the form of a function
\begin{eqnarray*}
SW_{M,\frak{s}}&:&\Bbb A(M)\rightarrow \Bbb Z\\ & & \ \ \ \alpha\longmapsto\langle \mu(\alpha),[\frak{M}]\rangle
\end{eqnarray*}
 where $\Bbb A(M)$ denotes the graded algebra obtained by tensoring the exterior algebra on $H_1(M;\Bbb Z)$ with grading one and the polynomial algebra on $H_0(M;\Bbb Z)$ with grading two, and the algebra homomorphism $\mu$ is defined as follows. For the positive generator $U$ of $H_0(M;\Bbb Z)$,
 $\mu(U)$ is the first Chern class of a principal $S^1$ bundle $\frak{M}_o$ over $\frak{M}$, where $\frak{M}_o$ is the solution space modulo the based gauge group $\mathcal{G}_o=\{h\in \mathcal{G}| h(o)=1\}$ for a fixed base point $o\in M$. For $[c]\in H_1(M,\Bbb Z)$,
$$\mu([c])=Hol_c^*([d\theta])$$ where $[d\theta]$ is the positive generator of $H^1(S^1,\Bbb Z)$, and  $Hol_c: \frak{M}\rightarrow  S^1$ is given by the holonomy of each connection around $c$.

Although  $SW_{M,\frak{s}}$ is a diffeomorphism invariant of $M$ for $b_2^+(M)> 1$,
when $b_2^+(M)= 1$, it depends on a chamber which is a connected component of $$\{\omega\in H^2(M;\Bbb R)-0\mid\omega^2\geq 0\}$$ so that Seiberg-Witten invariants may change according to which chamber the self-dual harmonic part of $-2\pi c_1(\frak{s})+\eta$ belongs to.
\begin{defn}
We call a Spin$^c$ structure $\frak{s}$ with $SW_{M,\frak{s}}(U^{\frac{d(\frak{s})}{2}})\ne 0$ a basic class of $M$, and  $M$ is called  of simple type, if $d(\frak{s})=0$ for any basic class $\frak{s}$ of $M$.
\end{defn}

One of major applications of Seiberg-Witten theory is the resolution of  the (generalized) Thom conjecture stating that a closed symplectic surface in a closed symplectic 4-manifold is genus-minimizing in its homology class. This is generalized to the adjunction inequality  on any 4-manifold with a nontrivial Seiberg-Witten invariant. On a smooth closed oriented 4-manifold $M$ of $b_2^+(M)>1$ and simple type, any embedded closed surface $\Sigma$ with genus $g(\Sigma)>0$  satisfies
$$|\langle [\Sigma],c_1(\frak{s})\rangle|+[\Sigma]\cdot [\Sigma]\leq 2g(\Sigma)-2$$ for any  basic class $\frak{s}$. If $[\Sigma]\cdot[\Sigma] \geq 0$, the simply type condition is unnecessary and moreover the inequality can be enhanced to the following.
\begin{thm}[P. Ozsv\'ath and Z. Szab\'o \cite{OS2}]\label{th1}
Let $M$ be a  smooth closed oriented 4-manifold  and $\Sigma\subset M$ be an embedded oriented surface with genus $g(\Sigma)>0$ representing a non-torsion homology class with $[\Sigma]\cdot[\Sigma]\geq 0$.

If $b_2^+(M)> 1$, then
$$|\langle [\Sigma],c_1(\frak{s})\rangle|+[\Sigma]\cdot [\Sigma]+(2-\min(b_1(M),1))d(\frak{s})\leq 2g(\Sigma)-2$$  for each basic class $\frak{s}$. If $b_2^+(M)= 1$, then for each basic class $\frak{s}$ with $$-\langle [\Sigma],c_1(\frak{s})\rangle+[\Sigma]\cdot [\Sigma]\geq 0,$$
$$-\langle [\Sigma],c_1(\frak{s})\rangle+[\Sigma]\cdot [\Sigma]+(2-\min(b_1(M),1))d(\frak{s})\leq 2g(\Sigma)-2,$$ where the Seiberg-Witten invariant is calculated in the chamber containing $PD[\Sigma]$.
\end{thm}

To state a more general version of Theorem \ref{th1}, recall that the first homology of a closed oriented surface can be viewed as a symplectic vector space given by the intersection pairing, and a basis for a symplectic vector space is called {\it symplectic} if the symplectic form takes the standard form in the basis. We define an invariant for an (embedded) oriented surface in a 4-manifold, which is a crucial tool in the present paper.
\begin{defn}
Let $M$ be a 4-manifold. For a closed oriented surface $\Sigma$ with genus $g>0$ embedded in $M$, define $l(\Sigma)$ to be the maximum of integers $l$ so that there is a symplectic basis $\{A_j,B_j\}^g_{j=1}$ in $H_1(\Sigma;\Bbb Z)$ satisfying that $i_*(A_j)=0$ in $H_1(M;\Bbb Q)$ for $j=1,\cdots,l$, where $i:\Sigma\rightarrow M$ is the inclusion map.
\end{defn}

\begin{thm}[P. Ozsv\'ath and Z. Szab\'o \cite{OS2}]\label{th2}
Let $M$ be a  smooth closed oriented 4-manifold of $b_2^+(M)>0$ and $\Sigma\subset M$ be an embedded oriented surface with genus $g(\Sigma)>0$ representing a non-torsion homology class with $\Sigma\cdot\Sigma\geq 0$.

Let $a\in \Bbb A(M)$ and $b\in \Bbb A(\Sigma)$ with degree $d(b)\leq l(\Sigma)$, and suppose $\frak{s}$ is a  Spin$^c$ structure with $SW_{M,\frak{s}}(a\cdot i_*(b))\ne 0$ (in the chamber containing $PD[\Sigma]$, if $b_2^+(M)= 1$).

If $b_2^+(M)> 1$, then $$|\langle [\Sigma],c_1(\frak{s})\rangle|+[\Sigma]\cdot [\Sigma]+2d(b)\leq 2g(\Sigma)-2.$$
If $b_2^+(M)= 1$ and $$-\langle [\Sigma],c_1(\frak{s})\rangle+[\Sigma]\cdot [\Sigma]\geq 0,$$  then
$$-\langle [\Sigma],c_1(\frak{s})\rangle+[\Sigma]\cdot [\Sigma]+2d(b)\leq 2g(\Sigma)-2.$$

Furthermore for $b$ with $d(b)>l(\Sigma)$, the similar inequalities hold with $2d(b)$ replaced with $d(b)$.
\end{thm}
Here the inclusion map $i:\Sigma\rightarrow M$ induces a map $i_*:\Bbb A(\Sigma)\rightarrow \Bbb A(M)$ for likewise defined $\Bbb A(\Sigma)$. From now on, $i_*$ will denote the homomorphism both on homologies and $\Bbb A(\cdot)$ induced by imbedding, and $d(b)$ will denote the degree of $b\in \Bbb A(\Sigma)$.
We improve the above theorem by replacing the condition involving $l(\Sigma)$ with a condition on $[\Sigma]$.
\begin{thm}\label{th3}
Let $M$ be a  smooth closed oriented 4-manifold of $b_2^+(M)>0$ and $\Sigma\subset M$ be an embedded oriented surface with genus $g(\Sigma)>0$ representing a non-torsion homology class with $[\Sigma]\cdot[\Sigma]\geq 0$.

Let $a\in \Bbb A(M)$ and $b\in \Bbb A(\Sigma)$, and suppose $\frak{s}$ is a Spin$^c$ structure with $SW_{M,\frak{s}}(a\cdot i_*(b))\ne 0$  (in the chamber containing $PD[\Sigma]$, if $b_2^+(M)= 1$).

If $b_2^+(M)> 1$ and  $$|\langle [\Sigma],c_1(\frak{s})\rangle|+[\Sigma]\cdot [\Sigma]\geq 2b_1(M),$$  then
$$|\langle [\Sigma],c_1(\frak{s})\rangle|+[\Sigma]\cdot [\Sigma]+2d(b)\leq 2g(\Sigma)-2.$$  If $b_2^+(M)= 1$ and   $$-\langle [\Sigma],c_1(\frak{s})\rangle+[\Sigma]\cdot [\Sigma]\geq 2b_1(M),$$ then
$$-\langle [\Sigma],c_1(\frak{s})\rangle+[\Sigma]\cdot [\Sigma]+2d(b)\leq 2g(\Sigma)-2.$$
\end{thm}
In case that $a=1$ and $b=U^{\frac{d(\frak{s})}{2}}$ where $U$ also denotes the positive generator of $H_0(\Sigma;\Bbb Z)$ from now on, this theorem generalizes Theorem \ref{th1} and it can be further extended to the following.
\begin{thm}\label{th4}
Let $M$ be a  smooth closed oriented 4-manifold of $b_2^+(M)>0$ and $\Sigma\subset M$ be an embedded oriented surface with genus $g(\Sigma)>0$ representing a non-torsion homology class with $[\Sigma]\cdot[\Sigma]\geq 0$. Suppose $\frak{s}$ is a basic class (in the chamber containing $PD[\Sigma]$, if $b_2^+(M)= 1$).

When $b_2^+(M)> 1$, if
\begin{eqnarray}\label{endure}
|\langle [\Sigma],c_1(\frak{s})\rangle|+3[\Sigma]\cdot [\Sigma]\geq 2b_1(M),
\end{eqnarray}
 then
$$|\langle [\Sigma],c_1(\frak{s})\rangle|+[\Sigma]\cdot [\Sigma]+2d(\frak{s})-2b_1(M)\leq 2g(\Sigma)-2.$$
When $b_2^+(M)= 1$, if
\begin{eqnarray}\label{help1}
-\langle [\Sigma],c_1(\frak{s})\rangle+3[\Sigma]\cdot [\Sigma]\geq 2b_1(M),
\end{eqnarray}
and
\begin{eqnarray}\label{help2}
-\langle [\Sigma],c_1(\frak{s})\rangle+[\Sigma]\cdot [\Sigma]\geq 0,
\end{eqnarray}
 then
$$-\langle [\Sigma],c_1(\frak{s})\rangle+[\Sigma]\cdot [\Sigma]+2d(\frak{s})-2b_1(M)\leq 2g(\Sigma)-2.$$
\end{thm}
We remark that Theorem \ref{th4} improves Theorem \ref{th1} only when $b_1(M)>0$ and $d(\frak{s})> 2$.

\section{Some algebraic lemmas}
\begin{lem}\label{loving}
Let $V$ be a symplectic vector space of dimension $2g\geq 4$ with a symplectic basis  $\{A_j,B_j\}^{g}_{j=1}$. Then for any integers $r$ and $s$, $\{A_j',B_j'\}^{g}_{j=1}$ where
$$A_1'=A_1-rA_3,\ \ \ A_2'=A_2-sA_3,\ \ \ B_3'=B_3+rB_1+sB_2,$$ and other $A_j'$ and $B_j'$ are the same as $A_j$ and $B_j$ respectively is also a symplectic basis of $V$.
\end{lem}
\begin{proof}
One can check it by a simple computation.
\end{proof}
Also note that since the above basis change is given by an integral symplectic matrix, its inverse is also an integral symplectic matrix. The following is our key lemma.
\begin{lem}\label{forgiveme}
Let $i:F\rightarrow M$ be an embedding of a closed oriented surface $F$ with genus $g>0$ into a 4-manifold $M$. Then any symplectic basis $\{A_j,B_j\}^{g}_{j=1}$ in $H_1(F;\Bbb Z)$ such that $$i_*(A_1)=\cdots =i_*(A_{l(F)})=0$$ in $H_1(M;\Bbb Q)$ satisfies that  the kernel of $$i_*:\Bbb Q\langle A_1,\cdots,A_{g} \rangle\rightarrow H_1(M;\Bbb Q)$$ has dimension $l(F)$, where  $\Bbb Q\langle A_1,\cdots,A_{g} \rangle$ is the $g$-dimensional $\Bbb Q$-vector subspace of $H_1(F;\Bbb Q)$ generated by $A_1,\cdots,A_{g}$.
\end{lem}
\begin{proof}
Let $\{A_j,B_j\}^{g}_{j=1}$ be a symplectic basis in $H_1(F;\Bbb Z)$, which realizes $l(F)$, i.e. $i_*(A_1)=\cdots =i_*(A_{l(F)})=0$ in $H_1(M;\Bbb Q)$.

When $g=1$, if $i_*(cA_1)$ for $c\in \Bbb Q-\{0\}$ is zero in $H_1(M;\Bbb Q)$, then so is  $i_*(A_1)$, and hence $\dim\ker i_*=l(F)$.

When $g\geq 2$, assume to the contrary that there exists a nonzero vector $v\in \Bbb Q\langle A_1,\cdots,A_{g} \rangle$ generated by $A_{l(F)+1},\cdots,A_{g}$ such that $i_*(v)=0$ in $H_1(M;\Bbb Q)$. By multiplying a rational number, if necessary, we may let $$v=\sum_{j=l(F)+1}^g a_jA_j$$ where $a_j$'s are integers such that their greatest common divisor is 1.

Now let's call the number of nonzero $a_j$'s $N$. The $N=1$ case is immediately excluded, because it is a contradiction to the definition of $l(F)$ as the maximum of $l$'s.

In the $N=2$ case, let's say $v=a_mA_m+a_nA_n$ for $\gcd(a_m,a_n)=1$. Take integers $p$ and $q$ such that $pa_m+qa_n=1$, and we modify the above symplectic basis by replacing $A_m,B_m,A_n,B_n$ with $$A_m'=v,\ \ \ \  B_m'=pB_m+qB_n,$$ $$A_n'=pA_n-qA_m,\ \ \ \  B_n'=a_mB_n-a_nB_m$$ respectively. One can easily check this new basis is still symplectic. But the fact that $i_*(A_m')$ is zero in $H_1(M;\Bbb Q)$ along with $i_*(A_j)$ for $j=1,\cdots,l(F)$ is again contradictory to the definition of $l(F)$.

For the higher $N$ cases, we will use induction on $N$. Suppose that $N\leq k$ cases lead to contradictions, and we need to prove for the $N=k+1$ case. Let's re-denote those nonzero $a_j$'s by $a_m, a_{m+1},\cdots, a_{m+k}$ such that $\gcd(a_m,a_{m+1})$ has the smallest value among all possible $\gcd(a_i,a_j)$ for $i\ne j$. There exist integers $r$ and $s$ such that $a_{m+2}':=a_{m+2}+ra_m+sa_{m+1}$ satisfies
\begin{eqnarray}\label{onlyword}
0\leq a_{m+2}'\leq \gcd(a_m,a_{m+1})-1.
\end{eqnarray}
Then we modify the symplectic basis by replacing $A_{m},A_{m+1},B_{m+2}$ with
$$A_m'=A_m-rA_{m+2}, \ \ \ \ A_{m+1}'=A_{m+1}-sA_{m+2},$$  $$B_{m+2}'=B_{m+2}+rB_m+sB_{m+1}$$ respectively.
By Lemma \ref{loving}, this new basis is symplectic, and $v$ can be expressed as $$v=a_mA_m'+a_{m+1}A_{m+1}'+a_{m+2}'A_{m+2}+\cdots$$ where $\cdots$ terms are the same as before. If $a_{m+2}'=0$, then it is reduced to the $N=k$ case, and otherwise we keep doing this process finite times until we make certain $a_j$ zero, because  (\ref{onlyword}) implies that $\min_{i\ne j}\gcd(a_i,a_j)$ is reduced at least by 1 whenever performing this symplectic basis change. This completes the proof.
\end{proof}
The above lemma can be rephrased as the following.
\begin{lem}\label{referee}
Under the assumptions of Lemma \ref{forgiveme},  $$g-b_1(M)\leq l(F).$$
\end{lem}
\begin{proof}
For a symplectic basis in $H_1(F;\Bbb Z)$ realizing $l(F)$, Lemma \ref{forgiveme} dictates that
 the homomorphism $i_*:\Bbb Q\langle A_1,\cdots,A_g \rangle\rightarrow H_1(M;\Bbb Q)$ satisfies
$$g-l(F)=\dim(\textrm{Im} (i_*))\leq b_1(M).$$
\end{proof}

\begin{thm}\label{key}
Let $M$ be a  smooth closed oriented 4-manifold of $b_2^+(M)>0$ and $\Sigma\subset M$ be an embedded oriented surface with genus $g(\Sigma)>0$ representing a non-torsion homology class with $[\Sigma]\cdot[\Sigma]\geq 0$.

Let $a\in \Bbb A(M)$ and $b\in \Bbb A(\Sigma)$, and suppose $\frak{s}$ is a Spin$^c$ structure with $SW_{M,\frak{s}}(a\cdot i_*(b))\ne 0$  (in the chamber containing $PD[\Sigma]$, if $b_2^+(M)= 1$).

Suppose that $$|\langle [\Sigma],c_1(\frak{s})\rangle|+[\Sigma]\cdot [\Sigma]\geq 2b_1(M),$$ when $b_2^+(M)> 1$, and $$-\langle [\Sigma],c_1(\frak{s})\rangle+[\Sigma]\cdot [\Sigma]\geq 2b_1(M),$$ when $b_2^+(M)= 1$.

Then $$d(b)\leq g(\Sigma)-b_1(M).$$
\end{thm}
\begin{proof}
Let's first consider the $b_2^+(M)>1$ case.
Assume to the contrary that $d(b)> g(\Sigma)-b_1(M)$. Let $\Sigma'\subset M$ be an embedded oriented surface obtained by adding $d(b)-g(\Sigma)+b_1(M)$ topologically trivial handles to $\Sigma$ so that $[\Sigma']=[\Sigma]$, and $g(\Sigma')=d(b)+b_1(M)$. Moreover $\Bbb A(\Sigma)$ naturally injects into $\Bbb A(\Sigma')$, and  $$d(b)=g(\Sigma')-b_1(M)\leq l(\Sigma')$$ by Lemma \ref{referee}.

Now we have
\begin{eqnarray*}
|\langle [\Sigma'],c_1(\frak{s})\rangle|+[\Sigma']\cdot [\Sigma']+2d(b)&=& |\langle [\Sigma],c_1(\frak{s})\rangle|+[\Sigma]\cdot [\Sigma]+2d(b)\\ &\geq& 2b_1(M)+2d(b)\\ &=& 2g(\Sigma')\\ &>& 2g(\Sigma')-2,
\end{eqnarray*}
which is a contradiction to Theorem \ref{th2} applied to $\Sigma'$.
Therefore $$d(b)\leq g(\Sigma)-b_1(M).$$

The case $b_2^+(M)=1$ can be proved in the same way as above by replacing $|\cdot|$ with a minus sign in  $|\langle [\Sigma],c_1(\frak{s})\rangle|$ and $|\langle [\Sigma'],c_1(\frak{s})\rangle|$.
\end{proof}

\begin{rmk}
Note that
$|\langle [\Sigma],c_1(\frak{s})\rangle|+[\Sigma]\cdot [\Sigma]$ is even for any Spin$^c$ structure $\frak{s}$ and any closed surface $\Sigma$ by the Wu formula.
\end{rmk}



\section{Proof of Theorem \ref{th3}}
By Lemma \ref{referee} and Theorem \ref{key}, we have $$d(b)\leq g(\Sigma)-b_1(M)\leq l(\Sigma).$$
Then the application of Theorem \ref{th2} gives the desired result.

\section{Proof of Theorem \ref{th4}}
\setcounter{equation}{0}

\begin{lem}\label{L1}
If an additional condition $[\Sigma]\cdot [\Sigma]\leq \min(b_1(M),\frac{d(\frak{s})}{2})$ is satisfied, then the desired adjunction inequalities hold.
\end{lem}
\begin{proof}
We use the blow-up technique as in \cite{kim, OS1, OS2}. Take $\hat{M}=M\# r\overline{\Bbb CP}_2$ for $r\geq 0$ and let $\hat{\Sigma}$ be the ``proper transform'' of $\Sigma$ so that $$[\hat{\Sigma}]=[\Sigma]-E_1-\cdots -E_r,$$ where $E_i$'s are the classes of exceptional spheres.

Let  $\hat{\frak{s}}$ be the Spin$^c$ structure on $\hat{M}$ which agrees with $\frak{s}$ in the complement of exceptional spheres and has 1st Chern class $$c_1(\hat{\frak{s}})=c_1(\frak{s})-3\sum_{i=1}^r \textrm{PD}[E_i].$$ By simple computations, $$[\hat{\Sigma}]\cdot [\hat{\Sigma}]=[\Sigma]\cdot[\Sigma]-r,$$ and $$d(\hat{\frak{s}})= d(\frak{s})-2r.$$

First, let's prove for the $b_2^+(M)>1$ case. Without loss of generality we may assume
$\langle [\Sigma],c_1(\frak{s})\rangle\leq 0$ by replacing $[\Sigma]$ with $-[\Sigma]$ if necessary.
Then
$$|\langle [\hat{\Sigma}],c_1(\hat{\frak{s}})\rangle|=|\langle [\Sigma],c_1(\frak{s})\rangle|+3r,$$ and
\begin{eqnarray}\label{love}
\\
|\langle [\hat{\Sigma}],c_1(\hat{\frak{s}})\rangle|+[\hat{\Sigma}]\cdot [\hat{\Sigma}]+2d(\hat{\frak{s}})&=&|\langle [\Sigma],c_1(\frak{s})\rangle|+[\Sigma]\cdot[\Sigma]+2d(\frak{s})-2r.\nonumber
\end{eqnarray}
 If we take $r=[\Sigma]\cdot[\Sigma]$, then $[\hat{\Sigma}]\cdot [\hat{\Sigma}]\geq 0$, $d(\hat{\frak{s}})\geq 0$, and
\begin{eqnarray*}
|\langle [\hat{\Sigma}],c_1(\hat{\frak{s}})\rangle|+[\hat{\Sigma}]\cdot [\hat{\Sigma}]&=&
|\langle [\Sigma],c_1(\frak{s})\rangle|+3[\Sigma]\cdot[\Sigma]\\ &\geq& 2b_1(M).
\end{eqnarray*}
By the well-known blow-up formula \cite{FS, sal} of Seiberg-Witten invariants,$$SW_{\hat{M},\hat{\frak{s}}}(U^{\frac{d(\hat{\frak{s}})}{2}})=SW_{M,\frak{s}}(U^{\frac{d(\frak{s})}{2}})\ne 0,$$ and hence $\hat{\frak{s}}$ is a basic class. We can now apply Theorem \ref{th3} to $\hat{M}$ with $a=1$ and $b=U^{\frac{d(\hat{\frak{s}})}{2}}$ to obtain
$$|\langle [\hat{\Sigma}],c_1(\hat{\frak{s}})\rangle|+[\hat{\Sigma}]\cdot [\hat{\Sigma}]+2d(\hat{\frak{s}})\leq 2g(\hat{\Sigma})-2= 2g(\Sigma)-2.$$
Combining this with (\ref{love}) and the assumption $b_1(M)\geq [\Sigma]\cdot[\Sigma]=r$, we get the desired adjunction inequality.

The proof for the $b_2^+(M)=1$ case proceeds in the same way as above by replacing $|\cdot|$ with a minus sign in  $|\langle [\hat{\Sigma}],c_1(\hat{\frak{s}})\rangle|$ and $|\langle [\Sigma],c_1(\frak{s})\rangle|$.
In this case, the blow-up formula says that $SW_{M,\frak{s}}(U^{\frac{d(\frak{s})}{2}})$ calculated in the chamber containing $\textrm{PD}[\Sigma]$ is equal to  $SW_{\hat{M},\hat{\frak{s}}}(U^{\frac{d(\hat{\frak{s}})}{2}})$ calculated in the chamber containing $\textrm{PD}[\hat{\Sigma}]$, which is what we need for the application of Theorem \ref{th3}.
\end{proof}

\begin{lem}\label{L2}
If an additional condition $b_1(M)\leq \min([\Sigma]\cdot [\Sigma],\frac{d(\frak{s})}{2})$ is satisfied, then the desired adjunction inequalities hold.
\end{lem}
\begin{proof}
For this lemma, we need neither (\ref{endure}) nor  (\ref{help1}).
The proof is similar to the previous lemma, and we adopt the same notation. Here we will take $r$ to be $b_1(M)$.

First let's consider the $b_2^+(M)>1$ case, and without loss of generality we may assume
$\langle [\Sigma],c_1(\frak{s})\rangle\leq 0$ by replacing $[\Sigma]$ with $-[\Sigma]$ if necessary.
By the assumption, we still have that $[\hat{\Sigma}]\cdot [\hat{\Sigma}]\geq 0$, $d(\hat{\frak{s}})\geq 0$, and
\begin{eqnarray*}
|\langle [\hat{\Sigma}],c_1(\hat{\frak{s}})\rangle|+[\hat{\Sigma}]\cdot [\hat{\Sigma}]&=&
|\langle [\Sigma],c_1(\frak{s})\rangle|+2r+[\Sigma]\cdot[\Sigma]\\ &\geq& 2b_1(M).
\end{eqnarray*}
Again by the blow-up formula \cite{FS, sal}, $\hat{\frak{s}}$ is a basic class, and hence Theorem \ref{th3} applied to $\hat{M}$ with $a=1$ and $b=U^{\frac{d(\hat{\frak{s}})}{2}}$  gives
$$|\langle [\hat{\Sigma}],c_1(\hat{\frak{s}})\rangle|+[\hat{\Sigma}]\cdot [\hat{\Sigma}]+2d(\hat{\frak{s}})\leq 2g(\hat{\Sigma})-2= 2g(\Sigma)-2.$$
Combining this with (\ref{love}), we get the desired adjunction inequality.

Again when $b_2^+(M)=1$, the proof goes through in the same way as the $b_2^+(M)>1$ case by replacing $|\cdot|$ with a minus sign in  $|\langle [\hat{\Sigma}],c_1(\hat{\frak{s}})\rangle|$ and $|\langle [\Sigma],c_1(\frak{s})\rangle|$.
\end{proof}

We divide the proof of Theorem  \ref{th4} into two cases according to whether $b_1(M)\leq \frac{d(\frak{s})}{2}$ or not.
Suppose the first case. If  $b_1(M)\geq [\Sigma]\cdot [\Sigma]$, then the proof is done by Lemma \ref{L1}, and if $b_1(M)\leq [\Sigma]\cdot [\Sigma]$, then the proof is given by Lemma \ref{L2}.

Now suppose $b_1(M)> \frac{d(\frak{s})}{2}$.
Then
$$|\langle [\Sigma],c_1(\frak{s})\rangle|+[\Sigma]\cdot [\Sigma]+2d(\frak{s})-2b_1(M)<
|\langle [\Sigma],c_1(\frak{s})\rangle|+[\Sigma]\cdot [\Sigma]+d(\frak{s}),$$ and hence when $b_2^+(M)> 1$, the RHS is less than or equal to  $2g(\Sigma)-2$ by Theorem \ref{th1}.
If $b_2^+(M)= 1$, we can also apply  Theorem \ref{th1} due to the condition (\ref{help2}), and hence we deduce that
\begin{eqnarray*}
-\langle [\Sigma],c_1(\frak{s})\rangle+[\Sigma]\cdot [\Sigma]+2d(\frak{s})-2b_1(M)&<&
-\langle [\Sigma],c_1(\frak{s})\rangle|+[\Sigma]\cdot [\Sigma]+d(\frak{s})\\ &\leq& 2g(\Sigma)-2.
\end{eqnarray*}
This completes the proof.

\bigskip

\noindent{\bf Acknowledgement.}  This work was supported by the National Research Foundation of Korea (NRF) grant funded by the Korea government (NRF-2014R1A1A4A01008987), and  it is the extended version of our earlier preprint, the anonymous referee of which we would like to sincerely thank for pointing out its incompleteness.

\end{document}